\documentclass[10pt,reqno]{amsart}
\usepackage[cp1251]{inputenc}
\usepackage[english]{babel}
\usepackage{amsmath}
\usepackage{amssymb}
\usepackage{amsfonts}
\usepackage{graphicx}
\usepackage{eucal}
\usepackage{hyperref,amsthm}
\RequirePackage{bbold}
\usepackage{mathptmx,bm}
\newtheorem{theorem}{Theorem}
\newtheorem{lemma}{Lemma}

\theoremstyle{definition}

\newtheorem{remark}{Remark}
\numberwithin{equation}{section}
\begin{document}

\title[{On asymptotic normality of the total progeny in the positive recurrent Q-processes}]
    {On asymptotic normality of the total progeny in the positive recurrent Q-processes}

\author{{Azam~A.~IMOMOV}}
\address {Azam Abdurakhimovich Imomov
\newline\hphantom{iii} Karshi State University,
\newline\hphantom{iii} 17, Kuchabag st., 180100 Karshi city, Uzbekistan.}
\email{{{imomov{\_}\,azam@mail.ru}}}

\author{{Zuhriddin~A.~NAZAROV}}
\address {Zuhriddin~A.~Nazarov
\newline\hphantom{iii} V.I.Romanovskiy Institute of Mathematics,
\newline\hphantom{iii} Uzbekistan Academy of Sciences, Tashkent 100174, Uzbekistan.}
\email{zuhrov13@gmail.com}

\thanks{\copyright \ 2023 Imomov~A.A, Nazarov~Z.A}



\dedicatory{Dedicated to the Fond Memory of our Great Ancestry}

\begin{abstract}
    We examine the population growth system called Q-processes. This is defined by the Galton-Watson Branching system
    conditioned on non-extinction of its trajectory in the remote future. In this paper we observe the total progeny
    up to time $n$ in the Q-process. By analogy with branching systems, this variable is of great interest in studying
    the deep properties of the Q-process. We find that the sum total progeny as a random variable approximates the standard
    normal distribution function under a second moment assumption for the initial Galton-Watson system offspring law.
    We estimate the speed rate of this approximation.
\end{abstract}

\maketitle

\section{Introduction and main results}      \label{ASec1}
    In the general theory of random processes models of stochastic branching systems are particularly important.
    Nowadays, there is great interest in these models. The creation of the theory of branching models is related to the
    possibility of estimating the survival probability of the population of monotypic individuals. The discrete-time simple
    branching process model was introduced by Francis Galton in 1889 as a  mathematical model for the population family growth
    is now called the Galton-Watson Branching (GWB) system; see \cite{AsHer}, \cite{ANey}, \cite{Harris63}, \cite{Jagers},
    \cite{Karp94}, \cite{Ken75} and \cite{Sev71}. GWB models play a fundamental role in both the theory and applications of
    stochastic processes. Among the random trajectories of branching systems, there are those that continue a long time.
    In the case of the GWB model, the class of such trajectories forms another stochastic model called Q-process;
    see \cite{ANey} and \cite{Imomov2014}. In the case of continuous-time Markov branching systems, an analogous
    model called the \emph{Markov Q-process}, was first introduced in \cite{Imomov2012}.

    Let $\left\{{Z(n), n\in{\mathbb{N}}_0}\right\}$ GWB system with branching rates $\left\{{p_k, k\in\mathbb{N}_0}\right\}$,
    where $\mathbb{N}_0=\{0\}\cup\mathbb{N}$ and $\mathbb{N}=\{1,2, \ldots\}$, the variable $Z(n)$ denote the population size at the moment $n$
    in the system. The evolution of the system occurs according to the following mechanism. Each individual lives a unit length life time
    and then gives $k\in\mathbb{N}_0$ descendants with probability $p_k$. This process is a reducible, homogeneous-discrete-time Markov
    chain with a state space consisting of two classes: $\mathcal{S}_0=\{0\}\cup\mathcal{S}$, where $\{0\}$ is absorbing state, and
    $\mathcal{S}\subset\mathbb{N}$ is the class of possible essential communicating states. Throughout the paper assume that $p_0>0$
    and $p_0+p_1>0$ which called the Schr\"{o}der case. We suppose that $p_0+p_1<1$ and $m:=\sum_{k\in\mathcal{S}}{kp_k}<\infty$.

    Considering transition probabilities
\begin{displaymath}
    P_{ij} (n):=\textsf{P}\left\{{Z(n+k)=j} \bigm\vert Z(k)=i \right\}
    \qquad {for \; any} \quad  k\in\mathbb{N}_0
\end{displaymath}
    we observe that the corresponding probability generating function (GF)
\begin{equation}               \label{AZ1}
    \sum_{k\in\mathcal{S}_0}{P_{ij}(n)s^k}=\bigl[{f_n(s)}\bigr]^i,
\end{equation}
    where $f_n(s):=\sum_{k\in\mathcal{S}_0}{\textsf{p}_k(n)s^k}$,
    therein $\textsf{p}_k(n):=P_{1k}(n)$ and, in the same time $f_n(s)$ is $n$-fold iteration of the offspring GF
\begin{displaymath}
    f(s):=\sum_{k \in\mathcal{S}_0}{p_k s^k}. 
\end{displaymath}    
    Needless to say that $f_n(0)=\textsf{p}_0(n)$ is a vanishing probability
    of the system initiated by one individual. Note that this probability tends as $n\to\infty$ monotonously to $q$, which
    called an extinction probability of the system, i.e. $\lim\nolimits_{n\to\infty}\textsf{p}_0(n)=q$; see \cite{ANey}.

    The extinction probability
\begin{itemize}
\item $q=1$ \quad if \; $m\leq1$;
\vspace{1.5mm}
\item $q<1$ \quad if \; $m>1$.
\end{itemize}
    Based on this, according to the values of the parameter $m$, the system is called
\begin{itemize}
\item \textit{sub-critical} \quad if \; $m<1$;
\vspace{1.5mm}
\item \textit{critical} \quad if \; $m=1$;
\vspace{1.5mm}
\item \textit{super-critical} \quad if \; $m>1$.
\end{itemize}

    Further we are dealing with the GWB system conditioned on the event $\left\{n<\mathcal{H}<\infty\right\}$, where
    $\mathcal{H}$ is an extinction time of the system, i.e. $\mathcal{H}:= \min\left\{{n\in\mathbb{N}: Z(n)=0}\right\}$.
    Let $\textsf{P}_{i}\bigl\{{*}\bigr\}:= \textsf{P}\left\{{{\ast} \bigm\vert {Z(0)=i}}\right\}$ and
    define conditioned probability measure
\begin{displaymath}
    \textsf{P}_i^{\mathcal{H}(n+k)}\{*\}: = \textsf{P}_i \left\{*\bigm\vert n+k<{\mathcal{H}<\infty}\right\}
    \qquad {for \; any} \quad  k\in\mathbb{N}.
\end{displaymath}
    In {\cite[p.~58]{ANey}} proved, that
\begin{equation}               \label{AZ2}
    \mathcal{Q}_{ij}(n):= \lim_{k \to \infty}
    \textsf{P}_i^{{\mathcal{H}}(n+k)}\bigl\{Z(n)=j\bigr\}={{jq^{j-i}}\over{i\beta^n}}P_{ij}(n),
\end{equation}
    where $\beta:=f'(q)$. Observe that ${\sum_{j\in{\mathbb{N}}}{\mathcal{Q}_{ij}(n)}}=1$ for each $i\in{\mathbb{N}}$. Thus,
    the probability measure $\mathcal{Q}_{ij}(n)$ can determine a new population growth system with the state space
    $\mathcal{E}\subset\mathbb{N}$ which we denote by $\left\{W(n), n\in{\mathbb{N}}_0\right\}$. This is
    a discrete-homogeneous-time irreducible Markov chain defined in the book {\cite[p.~58]{ANey}} and 
    called \textit{the Q-process}. Undoubtedly $W(0)\mathop=\limits^d Z(0)$ and transition probabilities
\begin{displaymath}
    {\mathcal{Q}}_{ij}(n) := \textsf{P}\left\{W(n)=j \Bigm\vert {W(0)=i}\right\}
    = \textsf{P}_i \left\{Z(n)=j \Bigm\vert {\mathcal{H}}=\infty\right\},
\end{displaymath}
    so that the Q-process can be interpreted as a ``long-living'' GWB system.

    Put into consideration a GF
\begin{displaymath}
    w_n^{(i)}(s):=\sum\limits_{j\in{\mathcal{E}}} {{\mathcal{Q}}_{ij}(n)s^j}.
\end{displaymath}
    Then from \eqref{AZ1} and \eqref{AZ2} we obtain
\begin{equation}              \label{AZ3}
    w_n^{(i)}(s)=\left[{{\frac{f_n(qs)} {q}}} \right]^{i-1} \cdot w_n(s),
\end{equation}
    where the GF $w_n(s):=w_n^{(1)}(s)=\textsf{E}\left[{s^{W(n)} \bigm\vert {W(0) =1}} \right]$ has a form of
\begin{equation}              \label{AZ4}
    w_n (s) = s{\frac{f'_n(qs)} {\beta^n}}
    \quad {for \; all} \quad n \in  {\mathbb{N}}.
\end{equation}
    Using iterations for $f(s)$ in \eqref{AZ3} leads to the following functional equation:
\begin{equation}              \label{AZ5}
    w_{n+1}^{(i)}(s)={\frac{w(s)} {f_{q}(s)}}w_n^{(i)}\bigl({f_{q}(s)}\bigr),
\end{equation}
    where $w(s):=w_1(s)$ and $f_{q}(s)= {f(qs)\big/ q}$. Thus, Q-process is completely defined by setting the GF
\begin{equation}              \label{AZ6}
    w(s) = s{\frac{f'(qs)} {\beta}}\raise 1.5pt\hbox{.}
\end{equation}

    An evolution of the Q-process is in essentially regulated by the structural parameter $\beta >0$. In fact,
    as it has been shown in \cite[p.~59,~Theorem~2]{ANey}, that
\begin{itemize}
\item $\mathcal{E}$ \textit{is positive recurrent} \quad if \; $\beta <1$;
\vspace{1.5mm}
\item $\mathcal{E}$ \textit{is transient} \quad if \; $\beta =1$.
\end{itemize}
    On the other hand, it is easy to be convinced that positive recurrent case $\beta <1$ of Q-process
    is in a definition character of the non-critical case $m\neq{1}$ of the initial GWB system.
    Note that $\beta \leq{1}$ and nothing but.

    In this paper we deal with the positive recurrent case assuming that first moment $\alpha:=w'(1-)$ be finite. Then
    differentiating \eqref{AZ6} on the point $s=1$ we obtain $\alpha = 1+\gamma_{q}\cdot\left({1-\beta}\right)$, where
\begin{displaymath}
    \gamma_{q}:={\frac{qf''(q)}{\beta\left(1-\beta\right)}} \raise 1.5pt\hbox{.}
\end{displaymath}
    It follows from \eqref{AZ3} and \eqref{AZ4} that
\begin{displaymath}
    {\textsf{E}}_i W(n):=\textsf{E}\left[{W(n) \Bigm\vert {W(0)=i}} \right]=\left({i-1} \right)\beta^n+\textsf{E}W(n),
\end{displaymath}
    where $\textsf{E}W(n)= 1+\gamma_{q}\cdot\bigl({1-\beta^n}\bigr)$.

    It is obvious, that when initial GWB system is sub-critical, then the condition $\alpha<\infty$ is
    this is equivalent to that $f''(1-)<\infty$.
    Further we everywhere will be accompanied by this condition by default.

    Our purpose is to investigate asymptotic properties of a random variable
\begin{displaymath}
    S_{n}=W(0)+W(1)+ \cdots +W(n-1),
\end{displaymath}
    denoting the total number of individuals that have existed up to the $n$-th generation in Q-process.
    By analogy with branching systems, this variable is of great interest in studying the deep properties
    of the Q-process. For details on the total progeny in GWB systems and related models
    results, see e.g. \cite{Karp94}, \cite{Ken75}, \cite{Kol84}, \cite{Pakes71}.

    Throughout this paper we will use famous Landau symbols ${o}$, ${\mathcal{O}}$ and ${\mathcal{O}^\ast}$ to describe kinds
    of bounds on asymptotic varying rates of positive functions $f(x)$ and $g(x)$. for for all large enough values of at infinity.
    So, $f={o}(g)$ means that $\lim_{x}{f(x)}\big/{g(x)}=0$, and we write $f={\mathcal{O}}(g)$ if
    $\limsup_{x}{f(x)}\big/{g(x)}<\infty$ and also we write $f={\mathcal{O}^\ast}(g)$
    if the ratio ${f(x)}\big/{g(x)}$ has a positive explicit limit. i.e. $\lim_{x}{f(x)}\big/{g(x)}=C<\infty$.
    Moreover, $f\sim(g)$ means that $\lim_{x}{f(x)}\big/{g(x)}=1$.

    Our main results are analogues of Central Limit Theorem and Law of Large Numbers for $S_{n}$.
    Let $\mathcal{N}\left({0,\sigma^{2}}\right)$ be a normal distributed random variable with the zero mean
    and the finite variance $\sigma^{2}$ and ${\Phi}_{0,\sigma^{2}}(x)$ is its distribution function.

\begin{theorem}             \label{ATh1}
    Let $\beta<1$ and $\alpha <\infty$. Then there exists a positive real-valued sequence
    $\mathcal{K}_{n}$ such that $\mathcal{K}_{n}=\mathcal{O}^{\ast}\left(\sqrt{n}\right)$ and
\begin{displaymath}
    \frac{S_{n}-\textsf{E}S_{n}}{\mathcal{K}_{n}} \buildrel{\textsf{P}}\over\longrightarrow
    \mathcal{N}\left({0,\sigma^{2}}\right) \qquad {as} \quad n\to\infty,
\end{displaymath}
    where the symbol ``$\buildrel{\textsf{P}}\over\longrightarrow$''
    means the convergence in probability.
\end{theorem}

\begin{theorem}             \label{ATh2}
    Let $\beta<1$ and $\alpha <\infty$. Then there exists slowly varying function at infinity $\mathcal{L}(\ast)$ such that
\begin{displaymath}
    \left|\textsf{P}\left\{\frac{S_{n}-\textsf{E}S_{n}}{\mathcal{K}_{n}}<x\right\}
    -{\Phi}_{0,\sigma^{2}}(x)\right|\leq {\frac{\mathcal{L}(n)}{n^{1/4}}}
\end{displaymath}
    uniformly in $x$.
\end{theorem}

    Let ${I_{a}}$ be a degenerate distribution concentrated at the point $a$, i.e.
\begin{displaymath}
    {I_{a}(B)}=
    \left\{\begin{array}{l}
    1 \qquad \hfill\text{if} \quad x\in{B},    \\
    \\
    0  \qquad \hfill\text{if} \quad x\notin{B}.
    \end{array} \right.
\end{displaymath}

\begin{theorem}             \label{ATh3}
    Let $\beta<1$ and $\alpha <\infty$. Then
\begin{displaymath}
    \frac{S_{n}}{n} \buildrel{\textsf{P}}\over\longrightarrow {1+\gamma_{q}} \qquad {as} \quad n\to\infty.
\end{displaymath}
    Moreover there exists slowly varying function at infinity ${\mathcal{L}_{\gamma}(\ast)}$ such that
\begin{displaymath}
    \left|\textsf{P}\left\{\frac{S_{n}}{n}<x\right\}
    -I_{1+\gamma_{q}}(x)\right|\leq {\frac{\mathcal{L}_{\gamma}(n)}{\sqrt{n}}}
\end{displaymath}
     uniformly in $x$, where
\begin{displaymath}
    {I_{1+\gamma_{q}}(x)}=
    \left\{\begin{array}{l}
    0 \qquad \hfill\text{if} \quad x\leq1+\gamma_{q},    \\
    \\
    1  \qquad \hfill\text{if} \quad x> 1+\gamma_{q}.
    \end{array} \right.
\end{displaymath}

\end{theorem}

    The rest of this paper is organized as follows.  Section~\ref{ASec2} provides auxiliary statements that will be essentially
    used in the proof of our theorems. Section~\ref{ASec3} is devoted to the proof of main results.

\section{Preliminaries}      \label{ASec2}

    Further we need the joint GF of the variables $W(n)$ and $S_{n}$
\begin{displaymath}
    J_{n}(s;x)=\sum_{j\in{\mathcal{E}}}^{}\sum_{l\in\mathbb{N}}^{}\textsf{P}\left\{W(n)=j, S_{n}=l\right\}s^{j}x^{l}
\end{displaymath}
    on a two-dimensional domain
\begin{displaymath}
    \mathbb{K}=\left\{(s;x)\in \mathbb{R}^{2}: \; s\in[0,1], \; x\in[0,1], \; \sqrt{(s-1)^{2}+(x-1)^{2}}>0\right\}.
\end{displaymath}
    Due to the Markov nature of the Q-process, we see that the two-dimensional one-step joint-transition probabilities
\begin{displaymath}
    \textsf{P}\left\{W(n+1)=j, S_{n+1}=l \Bigm\vert W(n)=i, S_{n}=k\right\}=
    \textsf{P}_{i}\left\{W(1)=j, S_{1}=l\right\}\delta_{l, i+k},
\end{displaymath}
    where $\delta_{ij}$ is the Kronecker's delta function:
\begin{displaymath}
    {\delta_{ij}}=
    \left\{\begin{array}{l}
    1  \qquad \hfill\text{if} \quad i=j,    \\
    \\
    0  \qquad \hfill\text{if} \quad i\neq{j}.
    \end{array} \right.
\end{displaymath}
    Therefore, we have
\begin{eqnarray*}
    \textsf{E}_{i}\left[s^{W(n+1)}x^{S_{n+1}} \Bigm\vert S_{n}=k\right]
    & = &\sum_{j\in\mathcal{E}}^{}\sum_{l\in \mathbb{N}}^{}\textsf{P}_{i}
    \left\{W(1)=j, S_{1}=l\right\}\delta_{l, i+k}s^{j}x^{l}   \nonumber\\
    \nonumber\\
    & = &\sum_{j\in\mathcal{E}}^{}\textsf{P}_{i}\left\{W(1)=j\right\}s^{j}x^{i+k}
    =w^{(i)}(s) x^{i+k}.
\end{eqnarray*}
    Next, using the formula of total probabilities, we obtain
\begin{eqnarray*}
    J_{n+1}(s;x)& = &
    \textsf{E}\left[\textsf{E}\left[s^{W(n+1)}x^{S_{n+1}} \Bigm\vert W(n), S_{n}\right]\right]
    =\textsf{E}\left[w^{W(n)}(s)  x^{W(n)+S_{n}}\right] \nonumber\\
    \nonumber\\
    & = &\textsf{E}\left[\bigl(w(s) f_{q}(s)\bigr)^{W(n)-1} x^{W(n)+S_{n}}\right]
    =\frac{w(s)}{f_{q}(s)} \textsf{E}\left[\bigl(xf_{q}(s)\bigr)^{W(n)}  x^{S_{n}}\right].
\end{eqnarray*}
    In the last line we  used formula \eqref{AZ3}. Thus we have
\begin{equation}    \label{AZ7}
    J_{n+1}(s;x)=\frac{w(s)}{f_{q}(s)}J_{n}\bigl(xf_{q}(s);x\bigr)
\end{equation}
    for $(s;x)\in \mathbb{K}$ and any $n\in \mathbb{N}$.

    Using relation \eqref{AZ7}, we can now obtain an explicit expression for the GF $J_{n}(s;x)$. Indeed,
    applying it consistently, taking into account \eqref{AZ6} and, after standard transformations, we have
\begin{equation}     \label{AZ8}
    J_{n}(s;x)=\frac{s}{\beta^{n}} \frac{\partial H_{n}(s;x)}{\partial{s}}\raise1pt\hbox{,}
\end{equation}
    where the function $H_{n}(s;x)$ is defined for any $(s;x)\in\mathbb{K}$
    by the following recursive relations:
\begin{equation}     \label{AZ9}
    \left\{\begin{array}{l}
    H_{0}(s;x)=s;    \\
    \\
    H_{n+1}(s;x)=xf_{q}\bigl(H_{n}(s;x)\bigr).
    \end{array} \right.
\end{equation}

    Since $\partial J_{n}(s;x)\bigl/\partial x\Bigl|_{(s;x)=(1;1)}=\textsf{E}S_{n}$,
    from \eqref{AZ8} and \eqref{AZ9}, we find that
\begin{equation}     \label{AZ10}
    \textsf{E}S_{n}=(1+\gamma_{q})n-\gamma_{q}\frac{1-\beta^{n}}{1-\beta} \raise1pt\hbox{.}
\end{equation}

\begin{remark}              \label{ARem1}
    Needles to say that the GF $f_{q}(s)=f(qs)\bigl/q$ generates a sub-critical GWB system. Denoting the population in this
    system as $Z_q(n)$, we define the sum $V_{n}=\sum_{k=0}^{n-1}Z_q(k)$ which is a total progeny of individuals that
    participated in the evolution of the system $\left\{Z_q(n), n\in\mathbb{N}_{0}\right\}$, up to the $n$-th generation.
    It is known that the GF of the joint distribution $\bigl(Z_q(n), V_{n}\bigr)$ satisfies the recursive equation \eqref{AZ9};
    see {\cite[p.~126]{Kol84}}.  Thus, the function $H_{n}(s;x)$ is a two-dimensional GF for all $n\in \mathbb{N}$ and
    $(s;x)\in \mathbb{K}$ and obeys to all properties of the GF $\textsf{E}\left[s^{Z_q(n)}x^{V_{n}}\right]$.
\end{remark}

    By virtue of what said in Remark~\ref{ARem1}, in studying $H_{k}(s;x)$ we use the properties of the GF
    $\textsf{E}\left[s^{Z_q(n)}x^{V_{n}}\right]$. Since the system $\left\{Z_q(n)\right\}$ is sub-critical,
    it goes extinct with probability $1$. Therefore, there exists a proper random variable $V={\lim}_{n\to\infty}V_{n}$,
    which means the total number of individuals participated in the whole evolution of the system. So
\begin{displaymath}
    h(x):=\textsf{E}x^{V}=\lim_{n\to\infty}\textsf{E}x^{V_{n}}=\lim_{n\to\infty}H_{n}(1;x)
\end{displaymath}
    and, according to \eqref{AZ9} it satisfies the functional equation
\begin{equation}     \label{AZ11}
    h(x)=xf_{q}\bigl(h(x)\bigr).
\end{equation}
    Further, we note that
\begin{displaymath}
    \textsf{P}\bigl\{Z_q(n)=0, V_{n}=k\bigr\}=\textsf{P}\bigl\{Z_q(n)=0, V=k \bigr\}.
\end{displaymath}
    Then, due to the monotonicity of the probabilistic GF, we find
\begin{displaymath}
    \textsf{P}\bigl\{V=k\bigr\}-\sum_{i\in\mathbb{N}}^{}\textsf{P}\bigl\{Z_q(n)=i, V_{n}=k\bigr\}s^{i}
    \leq \textsf{P}\bigl\{V=k, Z_q(n)>0\bigr\}.
\end{displaymath}
    Therefore, denoting
\begin{displaymath}
    R_{n}(s;x):=h(x)-H_{n}(s;x)
\end{displaymath}
    for $(s;x)\in\mathbb{K}$, we have
\begin{displaymath}
    R_{n}(s;x)\leq \sum_{k\in \mathbb{N}}^{}\textsf{P}\left\{V=k, Z_q(n)>0\right\}x^{k}=R_{n}(0;x).
\end{displaymath}
    It is easy to see $R_{n}(0;x)\leq R_{n}(0;1)=\textsf{P}\left\{Z_q(n)>0\right\}$. Then
\begin{equation}     \label{AZ12}
    \bigl|R_{n}(s;x)\bigr|\leq \textsf{P}\left\{Z_q(n)>0\right\}\longrightarrow 0 \qquad {as} \quad  n\to\infty.
\end{equation}
    On the other hand, due to the fact that $\left|h(x)\right|\leq 1$ and $\left|H_{n}(s;x)\right|\leq 1$ we have
\begin{eqnarray*}
    R_{n}(s;x) &=& x\left[f_{q}\bigl(h(x)\bigr)-f_{q}\bigl(H_{n-1}(s;x)\bigr)\right] \nonumber\\
    \\
    &=& x\textsf{E}\bigl[h(x)-H_{n-1}(s;x)\bigr]^{Z_q(n)} \leq \beta R_{n-1}(s;x)
\end{eqnarray*}
for all $(s;x)\in \mathbb{K}$. This implies that
\begin{equation}     \label{AZ13}
    \bigl|R_{n}(s;x)\bigr| \leq \beta^{n-k} \bigl|R_{k}(s;x)\bigr|
\end{equation}
    for any $n\in \mathbb{N}$ and $k=0,1,\dots,n$.

    In what follows, where the function $R_{n}(s;x)$ will be used, we deal with the domain $\mathbb{K}$, where this function does
    not vanish. By virtue of \eqref{AZ12},  taking into account \eqref{AZ9}, \eqref{AZ11}, we obtain the asymptotic formula
\begin{equation}     \label{AZ14}
    R_{n+1}(s;x)=xf_{q}'\bigl(h(x)\bigr)R_{n}(s;x)-
    x\frac{f_{q}''\bigl(h(x)\bigr)+\eta_{n}(s;x)}{2}R_{n}^{2}(s;x),
\end{equation}
    where $\left|\eta_{n}(s;x)\right|\to 0$ as $n\to\infty$ uniformly in $(s;x)\in \mathbb{K}$.
    Since $R_{n}(s;x)\to 0$, it follows from \eqref{AZ14} that
\begin{displaymath}
    R_{n}(s;x)=\frac{R_{n+1}(s;x)}{xf_{q}'\bigl(h(x)\bigr)}(1+o(1))  \qquad {as} \quad  n\to\infty.
\end{displaymath}
    Using last equality, we transform \eqref{AZ14} to the form
\begin{displaymath}
    R_{n+1}(s;x)=xf_{q}'\bigl(h(x)\bigr)R_{n}(s;x)-
    \left[\frac{f_{q}''\bigl(h(x)\bigr)}{2f_{q}'\bigl(h(x)\bigr)}+\varepsilon_{n}(s;x)\right]
    R_{n}(s;x)R_{n+1}(s;x)
\end{displaymath}
    and, therefore
\begin{equation}     \label{AZ15}
    \frac{u(x)}{R_{n+1}(s;x)}=\frac{1}{R_{n}(s;x)}+\upsilon(x)+\varepsilon_{n}(s;x),
\end{equation}
    where
\begin{displaymath}
    u(x)=xf_{q}'\bigl(h(x)\bigr) \qquad {and} \qquad \upsilon(x)=x\frac{f_{q}''\bigl(h(x)\bigr)}{2u(x)}
\end{displaymath}
    and $\sup_{(s;x)\in \mathbb{K}}\bigl|\varepsilon_{n}(s;x)\bigr|\leq \varepsilon_{n}\to 0$ as $n\to \infty$.
    By successively applying \eqref{AZ15}, we find the following representation for $R_{n}(s;x)$:
\begin{equation}     \label{AZ16}
    \frac{u^{n}(x)}{R_{n}(s;x)}=\frac{1}{R_{0}(s;x)}+\frac{\upsilon(x)\bigl[1-u^{n}(x)\bigr]}{1-u(x)}
    +\sum_{k=1}^{n}\varepsilon_{k}(s;x)u^{k}(x).
\end{equation}

    In what follows, our discussions will essentially be based on formula \eqref{AZ16}.
    Note that in the monograph {\cite[p.~136]{Kol84}} this formula was stated for the critical GWB system.

    Now, for convenience, we write
\begin{displaymath}
    J_{n}(s;x)=s\prod_{k=0}^{n-1}\frac{xf_{q}'\bigl(H_{k}(s;x)\bigr)}{\beta}
\end{displaymath}
    which is a direct consequence of formulas \eqref{AZ8} and \eqref{AZ9}.
    In our notation, it is almost obvious that $T_{n}(x):=\textsf{E}x^{S_{n}}=J_{n}(1;x)$. Then it follows that
\begin{equation}            \label{AZ17}
    T_{n}(x)=\prod_{k=0}^{n-1}u_{k}(x),
\end{equation}
    where
\begin{displaymath}
    u_{n}(x)=\frac{xf_{q}'\bigl(h_{n}(x)\bigr)}{\beta}\raise1pt\hbox{,}
\end{displaymath}
    at that $h_{n}(x)=\textsf{E}x^{V_{n}}$ which satisfies a recurrence equation $h_{n+1}(x)=xf_{q}\bigl(h_{n}(x)\bigr)$.
    Accordingly, the function $\Delta_{n}(x):=h(x)-h_{n}(x)$ satisfies the inequality
\begin{equation}            \label{AZ18}
    \bigl|\Delta_{n}(x)\bigr| \leq \beta^{n-k}\bigl|\Delta_{k}(x)\bigr|
\end{equation}
    which is a consequence of \eqref{AZ13}. Successive application of the inequality \eqref{AZ18} gives
\begin{equation}       \label{AZ19}
    \bigl|\Delta_{n}(x)\bigr|=\mathcal{O}\bigl(\beta^{n}\bigr)\to 0  \qquad {as} \quad  n\to\infty
\end{equation}
    uniformly in $x\in \mathbb{K}$. Similarly to the case $R_{n}(s;x)$,
    taking into account \eqref{AZ19} we find the following representation:
\begin{equation}     \label{AZ20}
    \frac{u^{n}(x)}{\Delta_{n}(x)}=\frac{1}{h(x)-1}+\frac{v(x)\bigl[1-u^{n}(x)\bigr]}{1-u(x)}+
    \sum_{k=1}^{n}\varepsilon_{k}(x)u^{k}(x),
\end{equation}
    where $\sup_{x\in \mathbb{K}}\bigl|\varepsilon_{n}(x)\bigr|\leq \varepsilon_{n}\to{0}$ as $n\to\infty$.

    In our further discussion we will also need expansions functions
    $h(x)$ and $u(x)$ in the left neighborhood of the point $x=1$.

\begin{lemma}               \label{ALem1}
    Let $\beta<1$ and $\alpha <\infty$. Then for GF $h(x)=\textsf{E}x^{V}$
    the following local expansion holds:
\begin{equation}      \label{AZ21}
    1-h(x) \sim \frac{1}{1-\beta}(1-x)-\frac{2\beta(1-\beta)+b_{q}}{2(1-\beta)^{3}}(1-x)^{2}  \qquad {as} \quad x\uparrow 1,
\end{equation}
    where $b_{q}:=f_{q}''(1-)$.
\end{lemma}

\begin{proof}
    We write the Peano's form Taylor expansion for $h(x)=\textsf{E}x^{V}$:
\begin{equation}      \label{AZ22}
    h(x)=1+h'(1-)(x-1)+\frac{h''(1-)}{2}(x-1)^{2}+o(x-1)^{2}  \qquad {as} \quad x\uparrow 1.
\end{equation}
    Formula \eqref{AZ11} and standard calculations produce that
\begin{displaymath}
    h'(1-)=\frac{1}{1-\beta} \qquad {and} \qquad h''(1-)=\frac{2\beta (1-\beta)+b_{q}}{(1-\beta)^{3}}\raise1.5pt\hbox{.}
\end{displaymath}
    Substituting these expressions in the expansion \eqref{AZ22}, entails \eqref{AZ21}.

    The lemma is proved.
\end{proof}

    Similar arguments can be used to verify the validity of the following lemma.

\begin{lemma}               \label{ALem2}
    Let $\beta<1$ and $\alpha <\infty$. Then
\begin{equation}        \label{AZ23}
    u(x)=\beta x\bigl[1-\gamma_{q}(1-x)\bigr]+ {\rho}(x),
\end{equation}
    where
\begin{displaymath}
    {\frac{{\rho}(x)}{(1-x)^2}}\to{const} \qquad {as} \quad x\uparrow 1.
\end{displaymath}
\end{lemma}

\begin{proof}
    Write the Taylor expansion with Lagrange error bound for $f'_{q}(y)$:
\begin{displaymath}
    f'_{q}(y)=\beta+f''_{q}(1)(y-1)+r(y),
\end{displaymath}
    where $r(y)\leq{A}\cdot(y-1)^{2}$ as $y\uparrow{1}$ and $A=const$.
    Since $u(x)=xf_{q}'\bigl(h(x)\bigr)$, taking herein $y=h(x)$ and using \eqref{AZ21} leads to \eqref{AZ23}.

    The lemma is proved.
\end{proof}

    The following two results directly follow from Lemma~\ref{ALem1} and Lemma~\ref{ALem2} respectively.

\begin{lemma}               \label{ALem3}
    Let $\beta <1$, $\alpha <\infty$. Then
\begin{equation}        \label{AZ24}
    h\bigl(e^{\theta}\bigr)-1 \sim \frac{\theta}{1-\beta}+
    \frac{2+\beta\gamma_{q}}{2(1-\beta)^{2}}\theta^{2} \qquad {as} \quad \theta\to{0}.
\end{equation}
\end{lemma}

\begin{lemma}               \label{ALem4}
    Let $\beta <1$, $\alpha <\infty$. Then
\begin{equation}        \label{AZ25}
    \frac{u\bigl(e^{\theta}\bigr)}{\beta}-1 = (1+\gamma_{q})\theta+{{\rho}(\theta)},
\end{equation}
    where ${{\rho}(\theta)}={\mathcal{O}^\ast}\left({\theta}^{2}\right)$ as $\theta\to{0}$.
\end{lemma}

    Next Lemma follows from combination of \eqref{AZ20}, \eqref{AZ24} and \eqref{AZ25}.
\begin{lemma}               \label{ALem5}
    Let $\beta <1$, $\alpha <\infty$. Then
\begin{equation}        \label{AZ26}
    \frac{\Delta_{n}\bigl(e^{\theta}\bigr)}{u^{n}(e^{\theta}\bigr)}
    = \frac{1}{1-\beta}\theta+{\mathcal{O}^\ast}\bigl({\theta^{2}}\bigr) \qquad {as} \quad \theta\to{0}
\end{equation}
    for any fixed $n\in \mathbb{N}$.
\end{lemma}

    Now we prove the following lemma.

\begin{lemma}               \label{ALem6}
    Let $\beta <1$, $\alpha <\infty$. Then
\begin{equation}        \label{AZ27}
    \ln\prod_{k=0}^{n-1}u_{k}\bigl(e^{\theta}\bigr) \sim -\left(1-\frac{u\bigl(e^{\theta}\bigr)}{\beta}\right)n-
    \gamma_{q}\theta \cdot \sum_{k=0}^{n-1}u^{k}\bigl(e^{\theta}\bigr) \qquad {as} \quad \theta\to{0}
\end{equation}
    for any fixed $n\in \mathbb{N}$.
\end{lemma}

\begin{proof}
    Using the inequality $\ln(1-y)\geq -y-y^{2}\bigl/(1-y)$, which is valid for $0\leq y<1$, we have
\begin{eqnarray}       \label{AZ28}
    \ln\prod_{k=0}^{n-1}u_{k}\bigl(e^{\theta}\bigr)
    &=& \sum_{k=0}^{n-1}\ln\left\{1-\left[1-u_{k}\bigl(e^{\theta}\bigr)\right]\right\}  \nonumber \\
    \nonumber \\
    &=& \sum_{k=0}^{n-1}\left[u_{k}\bigl(e^{\theta}\bigr)-1\right]
    +\rho_{n}^{(1)}(\theta) =: I_{n}(\theta)+\rho_{n}^{(1)}(\theta),
\end{eqnarray}
    where
\begin{equation}        \label{AZ29}
    I_{n}(\theta) = -\sum_{k=0}^{n-1}\left[1-u_{k}\bigl(e^{\theta}\bigr)\right],
\end{equation}
    and
\begin{displaymath}
    -\sum\limits_{k=0}^{n-1}\frac{\left[1-u_{k}\bigl(e^{\theta}\bigr)\right]^{2}}{u_{k}\bigl(e^{\theta}\bigr)}\leq\rho_{n}^{(1)}(\theta)\leq 0.
\end{displaymath}

    It is easy to see that the sequence of functions $\left\{h_{k}(x)\right\}$ does not decrease in $k\in \mathbb{N}$.
    Then, by the property of the GF, and the function $u_{k}\bigl(e^{\theta}\bigr)$ is non-decreasing in $k$, for any
    fixed $n\in \mathbb{N}$ and $\theta\in \mathbb{R}$. Therefore,
\begin{equation}        \label{AZ30}
    \frac{1-u_{0}\bigl(e^{\theta}\bigr)}{u_{0}\bigl(e^{\theta}\bigr)}I_{n}(\theta)\leq\rho_{n}^{(1)}(\theta)\leq 0.
\end{equation}
    According to the GF property, we will also verify that under our conditions 
    $1-u_{0}\bigl(e^{\theta}\bigr)\to 0$ as $\theta\to 0$. Then, according to \eqref{AZ30},
    $\rho_{n}^{(1)}(\theta)\to 0$ if only  $I_{n}(\theta)$ has a finite limit as $\theta\to 0$.

    Using the Taylor formula, we write
\begin{displaymath}
    f_{q}'(t)=f_{q}'(t_{0})-f_{q}''(t_{0})(t_{0}-t)+(t_{0}-t)g(t_{0};t),
\end{displaymath}
    where $g(t_{0};t)=(t_{0}-t)f_{q}'''(\tau)\bigl/2$ and $t_{0}<\tau<t$.
    Hence, at $t_{0}=h(x)$ and $t=h_{k}(x)$ we have the following relation:
\begin{displaymath}
    u_{k}(x)=\frac{u(x)}{\beta}-\frac{xf_{q}''\left(h(x)\right)}{\beta}\Delta_{k}(x)+\Delta_{k}(x)g_{k}(x),
\end{displaymath}
    where $g_{k}(x)=x\Delta_{k}(x)f_{q}'''(\tau)\bigl/2\beta$ and $h_{k}(x)<\tau<h(x)$. Therefore,
\begin{displaymath}
    u_{k}\bigl(e^{\theta}\bigr)=\frac{u\bigl(e^{\theta}\bigr)}{\beta}-
    \frac{e^{\theta}f_{q}''\left(h\bigl(e^{\theta}\bigr)\right)}{\beta}\Delta_{k}\bigl(e^{\theta}\bigr)
    +\Delta_{k}\bigl(e^{\theta}\bigr)g_{k}\bigl(e^{\theta}\bigr).
\end{displaymath}
    Then \eqref{AZ29} becomes
\begin{equation}       \label{AZ31}
    I_{n}(\theta)=-\left[1-\frac{u\bigl(e^{\theta}\bigr)}{\beta}\right]n-
    \frac{e^{\theta}f_{q}''\left(h\bigl(e^{\theta}\bigr)\right)}{\beta}
    \sum_{k=0}^{n-1}\Delta_{k}\bigl(e^{\theta}\bigr)+\rho_{n}^{(2)}(\theta),
\end{equation}
    where
\begin{displaymath}
    0 \leq \rho_{n}^{(2)}(\theta) \leq \Delta_{0}\bigl(e^{\theta}\bigr)\sum_{k=0}^{n-1}g_{k}\bigl(e^{\theta}\bigr).
\end{displaymath}
    In the last step we used the fact that $\bigl|\Delta_{n}(x)\bigr| \leq \beta^{n} \bigl|\Delta_{0}(x)\bigr|$ which follows from 
    inequality \eqref{AZ18}. It follows from \eqref{AZ24} that $\Delta_{0}\bigl(e^{\theta}\bigr)=\mathcal{O}(\theta)$ as $\theta\to{0}$. 
    And also the asymptotic estimation \eqref{AZ19} implies that $g_{k}\bigl(e^{\theta}\bigr)=\mathcal{O}\bigl(\beta^{k}\bigr)$
    as $k\to\infty$ and hence the functional series $\sum_{k=0}^{\infty}g_{k}\bigl(e^{\theta}\bigr)$ converges
    for all $\theta\in \mathbb{R}$. Therefore,
\begin{displaymath}
    \Delta_{0}\bigl(e^{\theta}\bigr)\sum_{k=0}^{n-1}g_{k}\bigl(e^{\theta}\bigr)
    =\mathcal{O}(\theta)\to{0} \qquad {as} \quad \theta\to{0}.
\end{displaymath}
    Then the remainder term in \eqref{AZ31}
\begin{equation}        \label{AZ32}
    \rho_{n}^{(2)}(\theta)\to{0} \qquad {as} \quad \theta\to{0}.
\end{equation}
    Assertion \eqref{AZ26} implies that
\begin{equation}       \label{AZ33}
    \sum_{k=0}^{n-1}\Delta_{k}\bigl(e^{\theta}\bigr)= \frac{\theta}{1-\beta}
    \sum_{k=0}^{n-1}u^{k}\bigl(e^{\theta}\bigr) \left(1+\mathcal{O}^{\ast}\bigl({\theta}\bigr)\right)
    \qquad {as} \quad \theta\to{0}.
\end{equation}
    Since ${e^{\theta}f''_{q}\left(h\bigl(e^{\theta}\bigr)\right)}\to{f''_{q}(1)}$ as $\theta\to{0}$,
    combining relations \eqref{AZ28}, \eqref{AZ31}--\eqref{AZ33} and, after some calculations,
    we will come to \eqref{AZ27}.

    The lemma is proved.
\end{proof}

\section{Proof of Theorems}             \label{ASec3}

\begin{proof}[\textbf{Proof of Theorem~\ref{ATh1}}]
    Define a sequence of variables
\begin{displaymath}
    {\zeta_{n}}:=\frac{S_{n}-\textsf{E}S_{n}}{\mathcal{K}_{n}}
\end{displaymath}
    for some positive real-valued sequence $\mathcal{K}_{n}$ such that $\mathcal{K}_{n}\to\infty$
    as $n\to\infty$ and then an appropriate characteristic function
\begin{displaymath}
    {\varphi_{\zeta_{n}}(\theta)}:=\textsf{E}\left[\exp\bigl\{i\theta\zeta_{n}\bigr\}\right]
    = \textsf{E}\left[\theta_{n}^{S_{n}} \cdot
    \exp\left\{\frac{-i\theta\textsf{E}S_{n}}{\mathcal{K}_{n}}\right\}\right],
\end{displaymath}
    where $\theta_{n}:=\exp\left\{i\theta \big/ \mathcal{K}_{n}\right\}$ and $\theta\in\mathbb{R}$.
    Using \eqref{AZ10} we write
\begin{equation}     \label{AZ34}
    \ln{\varphi_{\zeta_{n}}(\theta)} \sim -\left(1+\gamma_{q}\right)\frac{i\theta}{\mathcal{K}_{n}}n+
    \ln{T_{n}\left(\theta_{n}\right)} \qquad {as} \quad n\to\infty,
\end{equation}
    where $T_{n}(x)=\textsf{E}x^{S_{n}}$. Simultaneously according to \eqref{AZ17} and Lemma~\ref{ALem6},
\begin{equation}     \label{AZ35}
    \ln{T_{n}\left(\theta_{n}\right)} \sim -\left(1-\frac{u\left(\theta_{n}\right)}{\beta}\right)n-
    \frac{i\theta\gamma_{q}}{\mathcal{K}_{n}} \cdot\sum_{k=0}^{n-1}u^{k}\left(\theta_{n}\right)
\end{equation}
    as $n\to\infty$. In turn, \eqref{AZ25} implies
\begin{equation}     \label{AZ36}
    -\left(1-\frac{u\left(\theta_{n}\right)}{\beta}\right)n =
    \left(1+\gamma_{q}\right){\frac{i\theta}{\mathcal{K}_{n}}}n
    +n\rho\left({\frac{i\theta}{\mathcal{K}_{n}}}\right),
\end{equation}
    where $0<\lim_{\theta\to{0}}{\rho(\theta)\big/{\theta}^{2}}=:{C_\rho}<\infty$. Now we readily choose
\begin{equation}     \label{AZ37}
    \mathcal{K}_{n}=\mathcal{O}^{\ast}\left(\sqrt{n}\right) \qquad {as} \quad n\to\infty
\end{equation}
    which is equivalent to $\mathcal{K}_{n}\big/\sqrt{n}\to{C_{\mathcal{K}}}>0$. Hence we see that
\begin{equation}     \label{AZ38}
    n\rho\left({\frac{i\theta}{\mathcal{K}_{n}}}\right)\to -{K}{\theta^{2}} \qquad {as} \quad n\to\infty,
\end{equation}
    where $K:={C_\rho}\big/{C^{2}_{\mathcal{K}}}>0$. At the same time, since $u(x)=xf_{q}'\bigl(h(x)\bigr)$, in our assumptions
    we observe that $u(x)\leq\beta$ uniformly in $x\in[0,1]$. Therefore, one can choose $\varepsilon>0$ so desirably small that
\begin{displaymath}
\left|u^{k}\left(\theta_{n}\right)-\beta^{k}\right|\leq\varepsilon
\end{displaymath}
    for large enough $n$. This entails that ${\lim_{n\to\infty}}\sum_{k=0}^{n-1}u^{k}\left(\theta_{n}\right)$ converges
    uniformly in $\theta\in\mathbb{R}$. Eventually, after combination of asymptotic estimations \eqref{AZ35}--\eqref{AZ38},
    and denoting $\sigma^2:=2C_\rho$, the relation \eqref{AZ34} becomes
\begin{equation}     \label{AZ39}
    \ln{\varphi_{\zeta_{n}}(\theta)}=-{\frac{\sigma^{2}\theta^{2}}{2}}+ {\mathcal{K}_{n}(\theta)},
\end{equation}
    where ${\mathcal{K}_{n}(\theta)}=\mathcal{O}^{\ast}\left({i\theta}\big/{\mathcal{K}_{n}}\right)$
    as $n\to\infty$. Finally, we conclude that
\begin{displaymath}
    {\varphi_{\zeta_{n}}(\theta)}\longrightarrow\exp\left\{-\frac{\sigma^{2}\theta^{2}}{2}\right\}
    \qquad {as} \quad n\to\infty
\end{displaymath}
    for any fixed $\theta\in\mathbb{R}$. The assertion follows now from the continuity theorem for characteristic functions.

    Theorem~\ref{ATh1} is proved.
\end{proof}

\begin{proof}[\textbf{Proof of Theorem~\ref{ATh2}}]
    The relation \eqref{AZ39} and formal use of inequalities
\begin{displaymath}
    \left|e^{iy}\right|\leq{1} \qquad {and} \qquad  \bigl|e^{iy}-1-y\bigr|\leq{\frac{|y|^2}{2}}
\end{displaymath}
    imply
\begin{eqnarray}       \label{AZ40}
    \left|{\varphi_{\zeta_{n}}(\theta)}-{e^{-{\sigma^{2}\theta^{2}}/{2}}}\right|
    &\leq& \left|{e^{-{\sigma^{2}\theta^{2}}/{2}}}\right|
    \left|e^{\mathcal{K}_{n}(\theta)}-1\right|   \nonumber\\
    \nonumber\\
    &\leq& \left|e^{\mathcal{K}_{n}(\theta)}-1-{\mathcal{K}_{n}(\theta)}\right|+
    \bigl|{\mathcal{K}_{n}(\theta)}\bigr|  \nonumber\\
    \nonumber\\
    &\leq& \frac{\bigl[{\mathcal{K}_{n}(\theta)}\bigr]^{2}}{2}+\bigl|{\mathcal{K}_{n}(\theta)}\bigr|
\end{eqnarray}
    for all $n$. By definition we write
\begin{displaymath}
    {\mathcal{K}_{n}(\theta)}={C(n)}\frac{i\theta}{\mathcal{K}_{n}}  \raise 1.5pt\hbox{,}
\end{displaymath}
    where $\lim_{n\to{\infty}}{C(n)}=C<\infty$. Then, denoting
\begin{displaymath}
    F_n(x):=\textsf{P}\bigl\{{\zeta_{n}}<x\bigr\},
\end{displaymath}
    and using the estimation \eqref{AZ40}, we obtain the Berry-Esseen approximation bound {\cite[p.~538]{Feller68}} as follows:
\begin{eqnarray}       \label{AZ41}
    \Bigl|F_n(x)-{\Phi}_{0,\sigma^{2}}(x)\Bigr|
    &\leq& \frac{1}{\pi} \int\limits_{-T}^{T}\left|\frac{{\varphi_{\zeta_{n}}(\theta)}
    -{e^{-{\sigma^{2}\theta^{2}}/{2}}}}{\theta}\right|d\theta +\frac{24{M}}{\pi{T}} \nonumber \\
    \nonumber \\
    &\leq& \frac{2}{\pi} {\frac{C(n)}{\mathcal{K}_{n}}}T +\frac{24{M}}{\pi{T}}
\end{eqnarray}
    for all $x$ and $T>0$, where ${M}$ is such that $\Phi'_{0,\sigma^{2}}(x)\leq{M}$.
    It can be decidedly taken that ${M}={1}\bigl/{\sigma\sqrt{2\pi}}$.

    We let $T\to\infty$ and in the same time it is necessary to be $T=o\left(\sqrt{n}\right)$ since
    $\mathcal{K}_{n}=\mathcal{O}^{\ast}\left(\sqrt{n}\right)$. We can choose $T$ in general, in the
    form of $T=n^{\delta}{\mathcal{L}_{T}(n)}$, where $0<\delta<1/2$ and $\mathcal{L}_{T}(n)$
    slowly varies at infinity in the sense of Karamata. Then we reform \eqref{AZ41} as follows:
\begin{equation}       \label{AZ42}
    \Bigl|F_n(x)-{\Phi}_{0,\sigma^{2}}(x)\Bigr| \leq
    \frac{\mathcal{L}_{C}(n)}{n^{1/2-\delta}}+\frac{\mathcal{L}_{M}(n)}{n^{\delta}} \raise 1.5pt\hbox{,}
\end{equation}
    where
\begin{displaymath}
    \mathcal{L}_{C}(n):=\frac{2C(n)}{\pi} {\mathcal{L}_{T}(n)} \qquad {and} \qquad
    \mathcal{L}_{M}(n):=\frac{24M}{\pi} \frac{1}{\mathcal{L}_{T}(n)} \raise 1.5pt\hbox{.}
\end{displaymath}
    To come up to optimum degree of an estimation of approximation in \eqref{AZ42}, we would choose
    value of $\delta$ such that $(1/2-\delta)\delta$ has reached the maximum value for $\delta\in(0,1/2)$.
    It happens only in a unique case when $\delta=1/2-\delta$ or $\delta=1/4$. Thus \eqref{AZ42} becomes
\begin{displaymath}
    \Bigl|F_n(x)-{\Phi}_{0,\sigma^{2}}(x)\Bigr| \leq
    \frac{\mathcal{L}(n)}{n^{1/4}} \raise 1.5pt\hbox{,}
\end{displaymath}
    where ${\mathcal{L}(n)}={\mathcal{L}_{C}(n)}+{\mathcal{L}_{M}(n)}$ slowly varies at infinity.

    The theorem proof is completed.
\end{proof}

\begin{proof}[\textbf{Proof of Theorem~\ref{ATh3}}]
    First we will show that
\begin{equation}          \label{AZ43}
    \frac{S_{n}}{n} \buildrel{\textsf{P}}\over\longrightarrow {1+\gamma_{q}} \qquad {as} \quad n\to\infty.
\end{equation}
    Writing
\begin{displaymath}
    \eta_n:=\frac{S_{n}}{n}=\frac{\textsf{E}S_{n}}{n}+\frac{\mathcal{K}_{n}}{n}{\zeta_{n}},
\end{displaymath}
    and considering \eqref{AZ10}, we have
\begin{eqnarray}          \label{AZ44}
    {\varphi_{\eta_{n}}(\theta)}
    &:=& \textsf{E}\left[\exp\bigl\{i\theta\eta_{n}\bigr\}\right]  \nonumber \\
    \nonumber \\
    &=& e^{i\theta\left(1+\gamma_{q}\right)}\bigl[{\varphi_{\zeta_{n}}(\theta)}\bigr]^{{\mathcal{K}_{n}}/n}
    \left(1-{\frac{i\theta\gamma_{q}}{1-\beta}}\frac{1}{n}{\bigl(1-\beta^{n}\bigr)}\right),
\end{eqnarray}
    where ${\varphi_{\zeta_{n}}(\theta)}=\textsf{E}\left[\exp\bigl\{i\theta\zeta_{n}\bigr\}\right]$.
    Relation \eqref{AZ39} implies
\begin{displaymath}
    {\varphi_{\zeta_{n}}(\theta)}={e^{-{\sigma^{2}\theta^{2}}/{2}}}
    \left(1+\mathcal{O}^{\ast}\left({i\theta}\big/{\mathcal{K}_{n}}\right)\right)
    \qquad {as} \quad n\to\infty
\end{displaymath}
    and hence $\bigl[{\varphi_{\zeta_{n}}(\theta)}\bigr]^{{\mathcal{K}_{n}}/n}\to{0}$ as $n\to\infty$.
    Thus \eqref{AZ44} entails
\begin{displaymath}
    {\varphi_{\eta_{n}}(\theta)} \to {e^{i\theta\left(1+\gamma_{q}\right)}} \qquad {as} \quad n\to\infty.
\end{displaymath}
    According to the continuity theorem, this is sufficient for being of  \eqref{AZ43}.

    From \eqref{AZ44} we obtain
\begin{eqnarray*}
    \left|{\varphi_{\eta_{n}}(\theta)}-{e^{i\theta\left(1+\gamma_{q}\right)}}\right|
    &\leq& \left|\bigl[{\varphi_{\zeta_{n}}(\theta)}\bigr]^{{\mathcal{K}_{n}}/n}
    \left(1-{\frac{i\theta\gamma_{q}}{1-\beta}}\frac{1}{n}{\bigl(1-\beta^{n}\bigr)}\right)-1\right|   \nonumber\\
    \nonumber\\
    &\leq& \left|{\frac{i\theta\gamma_{q}}{1-\beta}}\frac{1}{n}{\bigl(1-\beta^{n}\bigr)}\right|.
\end{eqnarray*}
    We accounted in the last step that $|{\varphi_{\ast}(\theta)}|\leq1$ for any characteristic function.
    Now we can write the Berry--Esseen bound as follows:
\begin{eqnarray*}
    \Bigl|\textsf{P}\bigl\{\eta_{n}<x\bigr\}-I_{1+\gamma_{q}}(x)\Bigr|
    &\leq& \frac{1}{\pi} \int\limits_{-T}^{T}\left|\frac{{\varphi_{\eta_{n}}(\theta)}
    -{e^{i\theta\left(1+\gamma_{q}\right)}}}{\theta}\right|d\theta
    +\frac{24{M}_{\eta}}{\pi{T}} \nonumber \\
    \nonumber \\
    &\leq& \frac{\gamma_{q}}{\pi} {\frac{\bigl(1-\beta^{n}\bigr)}{1-\beta}}\frac{2T}{n}
    +\frac{24}{\pi{T}} \leq \frac{1}{\pi} {\frac{2\gamma_{q}}{1-\beta}}\frac{T}{n}
    +\frac{24}{\pi{T}} \raise 1.5pt\hbox{,}
\end{eqnarray*}
    where we put ${M}_{\eta}=1$ which is suitable for the degenerate distribution function.

    In this case we choose $T=n^{\delta}{\mathcal{L}_{T}(n)}$,
    where $0<\delta<1$ and $\mathcal{L}_{T}(n)$ slowly varies
    at infinity. Therefore
\begin{equation}       \label{AZ45}
    \Bigl|\textsf{P}\bigl\{\eta_{n}<x\bigr\}-I_{1+\gamma_{q}}(x)\Bigr| \leq
    \frac{\mathcal{L}_{\beta}(n)}{n^{1-\delta}}+\frac{\mathcal{L}_{1}(n)}{n^{\delta}} \raise 1.5pt\hbox{,}
\end{equation}
    where
\begin{displaymath}
    \mathcal{L}_{\beta}(n):=\frac{1}{\pi} {\frac{2\gamma_{q}}{1-\beta}}{\mathcal{L}_{T}(n)} \qquad {and} \qquad
    \mathcal{L}_{1}(n):=\frac{24}{\pi} \frac{1}{\mathcal{L}_{T}(n)} \raise 1.5pt\hbox{.}
\end{displaymath}
    We find $\delta=1/2$ and \eqref{AZ45} becomes
\begin{displaymath}
    \Bigl|\textsf{P}\bigl\{\eta_{n}<x\bigr\}-I_{1+\gamma_{q}}(x)\Bigr|  \leq
    \frac{\mathcal{L}_{\gamma}(n)}{n^{1/2}} \raise 1.5pt\hbox{,}
\end{displaymath}
    where ${\mathcal{L}_{\gamma}(n)}={\mathcal{L}_{\beta}(n)}+{\mathcal{L}_{1}(n)}$ slowly varies at infinity.

    The proof is completed.
\end{proof}


\begin{thebibliography}{99}
\bibitem{AsHer}
    {Asmussen~S. and Hering~H.}
    {Branching processes}. Birkh\"{a}user, Boston, 1983.

\bibitem{ANey}
    {Athreya~K.~B. and Ney~P.~E.}
    {Branching processes}. Springer, New York, 1972.

\bibitem{Feller68}
    {Feller~W.}
    {An Introduction to Probability Theory and its Applications}. vol.1. John Wiley \& Sons, 1968.

\bibitem{Harris63}
    {Harris~T.~E.}
    {The theory of branching processes}. Springer-Verlag, Berlin, 1963.

\bibitem{Imomov2012}
    {Imomov~A.~A.}
    {On Markov continuous time analogue of Q-processes.} Theory Prob. and Math. Stat., 2012, 84, 57--64.

\bibitem{Imomov2014}
    {Imomov~A.~A.}
    {Limit Theorem for the Joint Distribution in the Q-processes.}
    {Journal of Siberian Federal University. Mathematics and Physics}, 2014. v.~7(3), 289--296.

\bibitem{Jagers}
    {Jagers~P.}
    {Branching Progresses with Biological applications}.
    JW~\&~Sons, Pitman Press, GB, 1975.

\bibitem{Karp94}
    {Karpenko~A.~V. and Nagaev~S.~V.}
    {Limit theorems for the total number of descendents for the Galton-Watson branching process}.
    Theory Probab. Appl., 1994, 38, 433--455.

\bibitem{Ken75}
    {Kennedy~D.~P.}
    {The Galton-Watson process conditioned on the total progeny}.
    Jour. Appl. Prob., 1975, {\bf{12}}, 800--806.

\bibitem{Kol84}
    {Kolchin~V.~F.}
    {Random mappings}. Nauka, Moscow, 1984. (Russian)

\bibitem{Pakes71}
    {Pakes~A.~G.}
    {Some limit theorems for the total progeny of a branching process}.
    {Adv. App. Prob.}, 1971, {\bf3(1)}, 176--192.

\bibitem{Sev71}
    {Sevastyanov~B.~A.}
    {Branching processes}. Nauka, Moscow, 1971. (Russian)
\end{thebibliography}
\end{document}